\documentclass[11pt]{article}
\usepackage{enumerate,latexsym}
\usepackage{latexsym}
\usepackage{amsmath,amssymb}
\usepackage{times}

\usepackage[ansinew]{inputenc}
\usepackage{graphicx}
\usepackage{color}

\usepackage{amsthm}

\def\rth{\mathbb{R}^3}
\def\R{\mathbb{R}}
\def\g{\gamma}

\def\B{\mathbb{B}}
\def\N{\mathbb{N}}

\newcommand{\nc}{\newcommand}
\newcommand{\ben}{\begin{enumerate}}
\newcommand{\bit}{\begin{itemize}}
\newcommand{\een}{\end{enumerate}}
\newcommand{\eit}{\end{itemize}}
\newcommand{\wh}{\widehat}
\newcommand{\Int}{\mathrm{Int}}

\newcommand{\cD}{{\mathcal D}}
\newcommand{\cC}{{\mathcal C}}

\newcommand{\cR}{{\mathcal R}}

\newcommand{\cH}{{\mathcal H}}
\newcommand{\cF}{{\mathcal F}}

\newcommand{\wt}{\widetilde}

\newcommand{\cN}{\mathcal N}
\newcommand{\ov}{\overline}

\def\a{{\alpha}}

\def\t{{\theta}}

\def\g{{\gamma}}
\def\G{{\Gamma}}

\def\de{{\delta}}
\def\be{{\beta}}
\def\ve{{\varepsilon}}
\def\S{\Sigma}

\def\centerbmp#1#2#3{\vskip#2\relax\centerline{\hbox to#1{\special
    {bmp:#3 x=#1, y=#2}\hfil}}}

\newtheorem{theorem}{Theorem}[section]
\newtheorem{lemma}[theorem]{Lemma}
\newtheorem{proposition}[theorem]{Proposition}

\newtheorem{remark}[theorem]{Remark}
\newtheorem{ddescription}[theorem]{Description}
\newtheorem{property}[theorem]{Property}
\newtheorem{corollary}[theorem]{Corollary}
\newtheorem{definition}[theorem]{Definition}
\newtheorem{conjecture}[theorem]{Conjecture}

\newtheorem{assertion}[theorem]{Assertion}

\newtheorem{claim}[theorem]{Claim}

\newcommand{\ed}{\end{document}}
\nc{\bl}{\begin{lemma} }

\nc{\el}{\end{lemma} }

\nc{\bt}{\begin{theorem} }

\nc{\et}{\end{theorem} }
\newcommand{\rc}{ \renewcommand }

\rc{\v}{    \overset{\longrightarrow} }

    \newtheorem{theoremn}{Theorem}

\definecolor{b}{rgb}{.1,.1,.7}

\date{}

\begin{document}

\begin{title}
{The geometry of constant mean curvature surfaces in $\mathbb R^3$}
\end{title}

\begin{author}
{William H. Meeks III\thanks{This material is based upon
   work for the NSF under Award No. DMS-1309236.
   Any opinions, findings, and conclusions or recommendations
   expressed in this publication are those of the authors and do not
   necessarily reflect the views of the NSF.}
   \and Giuseppe Tinaglia\thanks{The second author was partially
   supported by
EPSRC grant no. EP/M024512/1}}
\end{author}
\maketitle



%
%
%

\begin{abstract}
We derive intrinsic curvature and radius estimates for compact disks embedded
in $\rth$ with nonzero constant mean curvature and apply these
estimates to study the global geometry of complete surfaces embedded
in $\rth$ with nonzero constant mean curvature.

\vspace{.3cm}

\noindent{\it Mathematics Subject Classification:} Primary 53A10,
   Secondary 49Q05, 53C42

\noindent{\it Key words and phrases:} Minimal surface,
constant mean curvature, curvature estimates. 
\end{abstract}
\maketitle


\section{Introduction.}
A longstanding  problem in classical surface theory
is to classify the complete, simply-connected surfaces embedded in $\rth$ with
constant mean curvature. In the case the surface is simply-connected and
compact, this classification
follows by work of either Hopf~\cite{hf1} in 1951 or of
Alexandrov~\cite{aa1} in 1956, who gave different proofs that
a round sphere is the only possibility.
In this paper we prove that if a complete, embedded simply-connected surface has nonzero
constant mean curvature, then it is compact.
\begin{theorem} \label{round} Complete,
simply-connected surfaces  embedded in
$\rth$ with nonzero constant mean curvature are compact, and thus are round spheres.
\end{theorem}
Theorem~\ref{round}, together with
results  of  Colding and Minicozzi~\cite{cm35}
and Meeks and Rosenberg~\cite{mr8} \nocite{bb1}
that  show that  the complete, simply-connected minimal surfaces
embedded in $\rth$ are planes
and helicoids, finishes the classification of complete
simply-connected surfaces embedded in $\rth$ with
constant mean curvature.

In the first section of this paper we observe how combining results in our previous papers, namely~\cite{mt8,mt7,mt9}, leads to the following
 intrinsic radius and curvature estimates for
embedded disks in $\rth$ with nonzero constant mean curvature, where
the {\em radius} of a compact Riemannian surface with
boundary is the maximum intrinsic distance of points in the surface to its boundary.

\begin{theorem}[Radius Estimates] \label{rest} There exists an
${\mathcal R}\geq \pi$ such that any compact disk embedded in
$\rth$ of constant mean curvature
$H>0$ has radius less than $\frac{{\mathcal R}}{H}$.
\end{theorem}

\begin{theorem}[Curvature Estimates] \label{cest} Given $\delta$, $\cH>0$,
there exists a  $K(\delta,\cH)\geq\sqrt{2}\cH$ such that
any compact  disk $M$ embedded in $\rth$ with
constant mean curvature $H\geq \cH$ satisfies
\[
\sup_{\large \{p\in {M} \, \mid \,
d_{M}(p,\partial M)\geq \delta\}} |A_{M}|\leq  K(\delta,\cH),
\]
where $|A_{M}|$ is the norm of the second fundamental form
and $d_{M}$ is the intrinsic distance function of $M$.
\end{theorem}

The radius estimate in Theorem~\ref{rest}  implies that a complete,
simply-connected surface $M$ embedded in $\rth$ with nonzero
constant mean curvature is compact.
In this way Theorem~\ref{round} follows from Theorem~\ref{rest}.

We wish to emphasize  to the reader that the curvature estimates for
embedded constant mean curvature disks given in
Theorem~\ref{cest}  depend {only} on the {\em lower} positive bound
$\cH$ for their mean curvature. Previous important examples
of curvature estimates for constant mean curvature surfaces, assuming certain geometric conditions,
can be found in the literature;
see for instance~\cite{boti1,boti3, cs1,cm23,cm35,rst1,sc3,ss1,ssy1,tin3,tin4}.

Our investigation
is inspired by the pioneering work of Colding and Minicozzi
in the minimal case~\cite{cm21,cm22,cm24, cm23};
however in the constant positive mean curvature setting this
description leads to the
existence of radius and curvature estimates.
Since the plane and the helicoid are complete simply-connected minimal surfaces
properly embedded in $\rth$, a radius estimate does not hold in
the minimal case. Moreover rescalings of a helicoid
give rise to a sequence of embedded minimal disks with arbitrarily large
norms of their second fundamental forms at points that can be arbitrarily far
from their boundary curves; therefore in the minimal setting,  curvature estimates
also do not hold.

For clarity of  exposition, we will call an oriented surface
$M$ immersed in $\rth$ an {\it $H$-surface} if it
is {\it embedded}, {\em connected}  and it
has {\it positive constant mean curvature $H$}. We will  call an
$H$-surface an {\em $H$-disk} if the $H$-surface is homeomorphic
to a closed  disk in the Euclidean plane.

The  next corollary is an immediate
consequence of Theorem~\ref{cest}.

\begin{corollary} \label{corinj1} If $M$ is a
complete $H$-surface with positive injectivity radius ${r_0}$, then
 $$\sup_M |A_M|\leq K(r_0,H).$$
\end{corollary}

As complete $H$-surfaces of bounded norm of the second fundamental form are properly embedded in $\rth$
by Theorem~6.1 in~\cite{mr7}, Corollary~\ref{corinj1} implies the next result.

\begin{corollary} \label{corinj2} A
complete $H$-surface with positive injectivity radius is properly embedded in $\rth$.
\end{corollary}

Since there exists an $\ve>0$ such that for any $C>0$, every complete immersed
surface $\S$ in $\rth$ with $\sup_{\S}|A_\S|<C$ has injectivity radius greater
than $\ve/ C$, Corollary~\ref{corinj1} also demonstrates that a necessary and sufficient condition
for an $H$-surface to have bounded norm of the second fundamental form is that it has positive injectivity radius.

\begin{corollary}
 A complete $H$-surface has positive injectivity radius
 if and only if it has bounded norm of the second fundamental form.
\end{corollary}

In Section~\ref{CY-problem} we  obtain curvature estimates
for $H$-surfaces that are annuli; these estimates
are analogous to
the curvature estimates in Theorem~\ref{cest} for $H$-disks
but necessarily must also depend on the
flux\footnote{See Definitions~\ref{def:flux}
and \ref{flux-annulus} for the definition of this flux.} of a given annulus.
We then apply these new curvature estimates to prove the next Theorem~\ref{annulus} on the
properness of complete $H$-surfaces of finite topology.
Earlier as the main result in~\cite{cm35}, Colding and Minicozzi proved the
similar theorem that complete  minimal surfaces
of finite topology embedded in $\rth$ are proper, thereby solving the classical
Calabi-Yau problem in the minimal setting.

\begin{theorem}~\label{annulus} A complete $H$-surface  with smooth
compact boundary (possibly empty) and finite topology has
bounded norm of the second fundamental form and  is properly embedded in $\rth$.
\end{theorem}

Theorem~\ref{annulus} shows that certain classical results for $H$-surfaces
hold when the hypothesis of ``properly embedded" is
replaced by the weaker hypothesis of ``complete and embedded."
For instance, in the seminal paper~\cite{kks1}, Korevaar, Kusner and Solomon proved
that the ends of a properly embedded
$H$-surface of finite topology in $\rth$ are asymptotic to the ends of
surfaces of revolution defined by Delaunay~\cite{de1} in 1841, and that if such a surface has
two ends, then it must be a Delaunay surface.
Earlier Meeks~\cite{me17} proved that a properly embedded
$H$-surface of finite topology in $\rth$
cannot have one end.  In particular, this last result together with Theorem~\ref{annulus}
 gives a  generalization of Theorem~\ref{round}.

The  theory developed
in this manuscript also provides key tools for understanding the geometry
of $H$-disks in a Riemannian three-manifold, especially in the case that
the manifold is locally homogeneous.  These generalizations and applications of the work presented here
will appear in our forthcoming paper~\cite{mt1}.
See~\cite{mt12} for applications of the present paper to obtain area estimates for
closed $H$-surfaces of fixed genus embedded in a flat 3-torus; see~\cite{cmt2,cmt1,rodt1} for examples 
that demonstrate that
Theorem~\ref{annulus} does not hold in the hyperbolic 3-space $\mathbb{H}^3$ when  $H\in [0,1)$ and 
in the Riemannian 
product $\mathbb{H}^2\times\mathbb R$ 
when $H\in [0,1/2)$.


\section{The Intrinsic Curvature and Radius Estimates.} \label{sec:4}

In~\cite{mt7} we proved the following extrinsic curvature and radius estimates for compact  disks embedded in $\rth$ with
constant mean curvature.

\begin{theorem}[Extrinsic Curvature Estimates] \label{ext:cest}
Given $\delta,\cH>0$, there exists a constant $K_0(\delta,\cH)$ such that
for any $H$-disk $\cD$ with $H\geq \cH$,
$${\large{\LARGE \sup}_{\large \{p\in \cD \, \mid \, d_{\rth}(p,\partial
\cD)\geq \delta\}} |A_\cD|\leq  K_0(\delta,\cH)}.$$
\end{theorem}

\begin{theorem}[Extrinsic Radius Estimates] \label{ext:rest}
There exists a constant ${\mathcal R}_{0}\geq \pi$ such that any $H$-disk
$\cD$ has extrinsic radius less than {\Large $\frac{{\mathcal R}_0}{H}$.}
In other words, for any point $p\in \cD$,
$$d_{\rth}(p,\partial \cD)<{\cR_0}/{H}.$$
\end{theorem}

Thus, Theorems~\ref{rest}
and \ref{cest} are immediate consequences of   a
chord-arc type result  from~\cite{mt8},
namely Theorem~\ref{thm1.1} below, and Theorems~\ref{ext:rest}
and \ref{ext:cest}.
Key ingredients in the proof of Theorem~\ref{thm1.1}
include  results in~\cite{mt9}
and the main theorem in~\cite{mt13}.
The results in~\cite{mt13} and~\cite{mt9} that are
needed to prove Theorem~\ref{thm1.1} below make use of
theorems and tools discussed in~\cite{mt7}. The results
in~\cite{mt7,mt8,mt13,mt9}  are inspired by and generalize the main
results of Colding and Minicozzi in~\cite{cm23,cm35} for minimal
disks to the case of $H$-disks.

\begin{definition} {\em Given a point $p$ on a compact surface
$\Sigma\subset \rth$, $\S (p,R)$ denotes the closure of
 the component of $\Sigma \cap \B(p,R)$ passing through $p$.}
\end{definition}

\begin{theorem}[Weak chord arc property] \label{thm1.1} There exists a $\delta_1 \in (0,
\frac{1}{2})$  such that the following holds.

Let $\S$ be an   $H$-disk in $\rth.$  Then for all closed
intrinsic balls $\ov{B}_\S(x,R)$ in $\S-
\partial \S$: \ben[1.]\item $\S (x,\delta_1 R)$ is a disk with piecewise smooth boundary
$\partial \Sigma(x,\delta_1 R)\subset \partial \B(\de_1R)$. \item
$
 \S (x, \delta_1 R) \subset B_\S (x, \frac{R}{2}).$
\een
\end{theorem}

We begin by applying Theorem~\ref{thm1.1} to prove the intrinsic radius estimate.

\begin{proof}[Proof of Theorem~\ref{rest}.]
Without loss of generality, fix $H=1$. Arguing by
contradiction,  if the radius estimates
were false then there would exist a sequence of
$1$-disks containing arbitrarily large
geodesic balls with centers at the origin
$\vec{0}\in \rth$. Let $\Sigma_n$ denote the sequence of such $1$-disks and let
$B_\Sigma (\vec 0, n)\subset \Sigma_n$
be the sequence of geodesic balls. Theorem~\ref{thm1.1} implies $B_\Sigma (\vec 0, n)\subset \Sigma_n$
contains a 1-disk centered at $\vec{0}$ of extrinsic radius $\delta_1n$. For $n$
large enough, this contradicts the
Extrinsic Radius Estimate and completes the proof of Theorem~\ref{rest}.
\end{proof}

We next prove the  intrinsic curvature estimate.

\begin{proof}[Proof of Theorem~\ref{cest}]
  Let $\ve=\delta_1\delta$,
where $\delta_1\in (0,\frac{1}{2})$ is given in Theorem~\ref{thm1.1} and let
$K(\delta,\cH):=K_0(\ve,\cH)$, where $K_0{(\ve,\cH)}$ is given in
Theorem~\ref{ext:cest}. Let $\cD$ be an $H$-disk
with $H\geq\mathcal H$ and let $p\in \cD$ be a point
with $d_\cD(p,\partial \cD)\geq \delta$. By Theorem~\ref{thm1.1}, the closure of the
component   $E$ of $\cD\cap {\B}(p,\ve)$  containing $p$ is an $H$-disk in the
interior of $\cD$ with $\partial E\subset \B(p,\ve)$. By Theorem~\ref{ext:cest},
\[
|A_E|(p)\leq K_0(\ve,\cH)=K(\delta,\cH).
\]
 This completes the proof of
Theorem~\ref{cest}.
\end{proof}

\section{Curvature estimates
for $H$-annuli and properness of H-surfaces with finite topology.} \label{CY-problem}
A classical conjecture in the global theory of minimal surfaces,
first stated by Calabi in 1965~\cite{ca4} and later revisited by
Yau~\cite{yau1, yau2},  is the following:
\begin{conjecture}[Calabi-Yau Conjecture] There do not exist  complete
immersed minimal surfaces in a bounded domain in $\rth$.  \end{conjecture}
Based on earlier work of Jorge and Xavier~\cite{jx1},
Nadirashvili~\cite{na1} proved the existence of
 a complete, bounded, immersed minimal surface in
$\rth$, thereby disproving the  above conjecture.
In contrast to these results,
Colding and Minicozzi proved in~\cite{cm35} that complete, finite topology minimal
surfaces {\it embedded} in $\rth$ are proper. Thus, the  Calabi-Yau conjecture holds
in the classical setting of
complete, embedded, finite topology minimal
surfaces.

In this section we will apply Proposition~\ref{assertion} below to
obtain Theorem~\ref{annulus}, a result that
generalizes the  properness result
of Colding and Minicozzi for embedded minimal
surfaces of finite topology to the setting of $H$-surfaces.
In the proof of this proposition we will need to
apply the main theorems in~\cite{mt14}, whose proofs depend on
the results in the first three sections of the present paper, as well
as results in~\cite{mt8,mt13,mt9}.

Recall first the definition of
flux of a 1-cycle in an $H$-surface; see for instance~\cite{kks1,ku2,smyt1}
for further discussion of this invariant.

\begin{definition} \label{def:flux} {\em
Let $\gamma$ be a piecewise-smooth 1-cycle in an $H$-surface $M$. The  flux of
$\gamma$ is $\int_{\gamma}(H\gamma+\xi)\times \dot{\gamma}$, where $\xi$
is the unit normal to $M$ along $\gamma$ and $\gamma $ is
parameterized by arc length. }
\end{definition}

The flux is a homological invariant and
we say that $M$ has {\em  zero flux} if the flux of any 1-cycle in $M$ is zero;
in particular, since the first homology group of a disk is
zero, the flux of an $H$-disk is zero.

\begin{definition} \label{flux-annulus} {\em
Let $E$ be a compact $H$-annulus.  Then the {\em flux} \,$F(E)$ of $E$ is the length
of the flux vector of either  generator of
the first homology group of $E$.}
\end{definition}

The next proposition implies that given a compact $1$-annulus
with a fixed positive (or zero) flux, then given $\de>0$, the injectivity
radius function on this annulus is bounded  away from zero
at points of distance greater than $\de$ from its boundary.

\begin{proposition} \label{assertion} Given $\rho>0$ and $\delta\in (0,1)$
there exists a positive constant $I_0(\rho,\delta)$
  such that if $E$ is
a compact 1-annulus with $F(E)\geq \rho$  or with $F(E)=0$,
then $$\inf_{\{p\in E \; \mid \; d_E(p,\partial E)\geq \delta\}} I_E\geq I_0(\rho,\delta), $$
where $I_E\colon E\to [0,\infty)$ is the injectivity radius function of $E$.
\end{proposition}
\begin{proof} Arguing by contradiction, suppose there exist a $\rho>0$ and a
sequence $E(n)$ of compact 1-annuli satisfying $F(E(n))\geq \rho>0 $
or $F(E(n))=0$, with injectivity radius functions
$I_n\colon E(n)\to [0,\infty)$ and  points $p(n)$ in
$\{q\in E(n)\mid d_{E(n)}(q,\partial E(n))\geq \delta\}$ with
\[
I_n(p(n))\leq\frac{1}{n}.
\] We next use the fact that the injectivity radius function on a
complete Riemannian manifold with boundary is continuous.

 For each $p(n)$ consider a point
$q(n)\in \ov{B}_{E(n)}(p(n),\delta/2)$ where the following positive continuous
function obtains its maximum value:
\[
f\colon  \ov{B}_{E(n)}(p(n),\delta/2)\to (0,\infty),
\]
\[
f(x)=\frac{d_{E(n)}(x,\partial \ov{B}_{E(n)}(p(n),\delta/2))}{I_n(x)}.
\]
Let $r(n)=\frac{1}{2}d_{E(n)}(q(n),\partial \ov{B}_{E(n)}(p(n),\delta/2))$
and note that
\[
\frac{\delta/2}{I_n(q(n))}\geq\frac{2r(n)}{I_n(q(n))}=f(q(n))\geq f(p(n))\geq n\de/2.
\]
Moreover, if $x\in \ov{B}_{E(n)}(q(n),r(n))$, then by the triangle inequality,
\[
\frac{r(n)}{I_n(x)}\leq \frac{d_{E(n)}(x,
\partial \ov{B}_{E(n)}(p(n),\delta/2))}{I_n(x)}=f(x)\leq f(q(n))=\frac{2r(n)}{I_n(q(n))}.
\]
 Therefore, for $n$ large the $H_n$-surfaces
 $M(n)=\frac{1}{I_n(q(n))}[\ov{B}_{E(n)}(q(n),r(n))-q(n)]$ satisfy the following conditions:

 \begin{itemize}
\item $I_{M(n)}(\vec{0})=1$;
\item $d_{M(n)}(\vec{0},\partial M(n))\geq \frac{n\delta}{4}$;
\item  $I_{M(n)}(x)\geq \frac12$  for any $x\in \ov{B}_{M(n)}(\vec  0, \frac{n\delta}{4})$.
\end{itemize}

By Theorem~3.2 in~\cite{mt8}, for any $k\in \N$, there exists an $n(k)\in \N$ such
that
the closure of the component
$\Delta(n(k))$ of $M(n(k))\cap \B(k)$ containing the origin is a compact $H_{n(k)}$-surface
with boundary in $\partial \B(k)$ and the injectivity radius function of
$\Delta(n(k))$ restricted to  points in
$\Delta(n(k))\cap\B(k-\frac12)$ is at least $\frac12$.
By Theorem~1.3 of~\cite{mt14}, for $k$ sufficiently
large, $\Delta(n(k))$  contains a simple closed curve $\G(n(k))$
with the length of its nonzero flux vector bounded from above by some
constant $C>0$.
 Since the curves $\G(n(k))$ are rescalings of simple closed curves
$\wt{\G}(n(k))\subset E(n(k))$,  then the $\wt{\G}(n(k))$ are simple closed
curves with  nonzero flux. Hence these simple closed curves  are generators of the
first homology group of the annuli  $E(n(k))$. This immediately gives a contradiction in the case
that $F(E(n(k)))=0$.
If  $F(E(n(k)))\geq \rho>0$, we have that
\[
C\geq |F({\G}(n(k)))|=|F\left(\frac{1}{I_{n(k)}(q(n(k)))}\wt{\G}(n(k))\right)|
=\frac{|F(\wt{\G}(n(k)))|}{I_{n(k)}(q(n(k)))}
\]
\[
=\frac{F(E(n(k))}{I_{n(k)}(q(n(k)))}
\geq \frac{\rho}{I_{n(k)}(q(n(k)))}\geq\rho n(k).
\]
These inequalities  lead to  a contradiction for $n(k)>\frac{C}{\rho}$,
which completes the proof of the proposition.
\end{proof}

An immediate consequence of Proposition~\ref{assertion}
and the intrinsic curvature estimates for
$H$-disks
is the following result.

\begin{corollary} \label{assertion2} Given $\rho>0$ and $\delta\in (0,1)$
there exists a positive constant $A_0(\rho,\delta)$ such that if $E$ is
a compact 1-annulus with $F(E)\geq \rho$ or  with $F(E)=0$, then
\[
\sup_{\{p\in E\ \mid \, d_E(p,\partial E)\geq \delta\}} |A_E|\leq A_0(\rho,\delta).
\]
\end{corollary}

When $M$ has finite topology, the flux of each of its finitely many annular
ends is either zero or bounded away
from zero by a fixed positive number. Thus, Proposition~\ref{assertion} implies that the
injectivity radius function of $M$ is positive, and so the
norm of its second fundamental is bounded by Theorem~\ref{cest}.
The next corollary is a consequence of this last property and  the fact that a complete embedded
$H$-surface of bounded norm of the second fundamental form is properly embedded in
$\rth$; see Theorem~6.1 in~\cite{mr7} or item~1
of Corollary~2.5 in~\cite{mt3} for this properness
result.

\begin{corollary} \label{cor:finitetop}
A complete surface $M$ with finite topology embedded
in $\rth$ with nonzero constant mean
curvature has bounded norm of the second fundamental form and is properly embedded in $\rth$.
\end{corollary}

\begin{remark} \label{thm1.6}  {\em With slight modifications, the
proof of the above corollary generalizes to the case where the
$H$-surface $M$ above
is allowed to have smooth compact boundary;
 for example, see~\cite{mt4} for these types of  adaptations.
 Thus, Theorem~\ref{annulus} holds as well.}
 \end{remark}
Corollary~\ref{cor:finitetop} motivates our
conjecture below concerning the properness and at most cubical area growth estimates for
complete surfaces $M$ embedded in $\rth$ with finite
genus and constant mean curvature.

\begin{conjecture}\label{conjCY} {\em
A complete surface $M$ of finite genus embedded in $\rth$ with constant mean curvature
has at most cubical area growth in the sense that
such an $M$ has area less than $CR^3$ in ambient balls of radius $R$ for some $C$ depending on $M$.
In particular every such surface is properly embedded in $\rth$.
}\end{conjecture}

Conjecture~\ref{conjCY} holds for
complete minimal surfaces embedded in $\rth$
with a countable number of ends and finite genus. This cubical volume growth result follows from the properness
of such minimal surfaces (by Meeks, Perez and Ros in~\cite{mpr9}), because
properly embedded minimal surfaces in $\rth$  of finite genus
have bounded norm of the second fundamental form (by  Meeks, Perez and Ros in~\cite{mpr3}) and because
properly embedded minimal surfaces in $\rth$
with bounded norm of the second fundamental form have at most cubical volume growth (by Meeks and Rosenberg in~\cite{mr7}).

\vspace{.3cm}
\center{William H. Meeks, III at profmeeks@gmail.com\\
Mathematics Department, University of Massachusetts, Amherst, MA 01003}
\center{Giuseppe Tinaglia at giuseppe.tinaglia@kcl.ac.uk\\ Department of
Mathematics, King's College London,
London, WC2R 2LS, U.K.}

\bibliographystyle{plain}
\bibliography{bill}

\end{document}